\documentclass[11pt, a4paper]{article}

\usepackage{graphicx,tikz,caption,subcaption,fontenc}
\usepackage{placeins}
\usetikzlibrary{backgrounds}
\usepackage{fullpage}
\usepackage{enumitem,verbatim}
\usepackage{hyperref}
\usepackage{amsmath}
\usepackage{amssymb}
\usepackage{amsfonts}
\usepackage{amsthm}
\usepackage[capitalise]{cleveref}
\usepackage{bm}

\usepackage[margin=2.5cm]{geometry}

\newtheorem{theorem}{Theorem}[section]
\newtheorem{lemma}[theorem]{Lemma}
\newtheorem{claim}[theorem]{Claim}

\newtheorem{conjecture}[theorem]{Conjecture}

\newtheorem*{observation*}{Observation}
\newtheorem{proposition}[theorem]{Proposition}
\newtheorem{problem}[theorem]{Problem}

\newtheorem*{question*}{Question}
\newtheorem{innerdefinition}{Definition}
\newenvironment{definition*}
  {
   \innerdefinition}
  {\endinnerdefinition}

\theoremstyle{definition}
\newtheorem{defn}[theorem]{Definition}

\theoremstyle{remark}

\newcommand{\ex}{\mathrm{ex}}

\usepackage{tikz}
\tikzstyle{aNode} = [circle, fill = black]
\tikzstyle{bNode} = [circle,draw = black, thick]

\usetikzlibrary{patterns}
\usetikzlibrary{shapes}
\usetikzlibrary{decorations.pathreplacing}
\usetikzlibrary{decorations.pathmorphing}
\usetikzlibrary{positioning}
\usetikzlibrary{fit}
\usetikzlibrary{calc}

\usepackage{todonotes}

\title{Tree suspensions and transfer functions for single degree Tur\'an spectra}
\author{Jiangdong Ai\thanks{School of Mathematical Sciences and LPMC, Nankai University, Tianjin 300071, China. Supported by the National Natural Science Foundation of China (No.12522117). Email:
\texttt {jd@nankai.edu.cn}.}
\and
Laihao Ding\thanks{School of Mathematics and Statistics, and Key Laboratory of Nonlinear Analysis \& Applications (Ministry of Education), Central China Normal University, Wuhan 430079, China. Supported by the Fundamental Research Funds for the Central Universities (CCNU25JCPT031). Email: \texttt{dinglaihao@ccnu.edu.cn}.}
\and 
Hong Liu\thanks{Extremal Combinatorics and Probability Group (ECOPRO),  Institute for Basic Science (IBS), Daejeon, South Korea. Supported by IBS-R029-C4. Email: \texttt{hongliu@ibs.re.kr}.}
\and Haotian Yang\thanks{Department of Mathematics, California Institute of Technology, Pasadena, USA. Supported by Seed Fund Program for International Research Cooperation of Shandong University. Email: \texttt{hyang3@caltech.edu}.}}

\begin{document}
\date{}
\maketitle

\begin{abstract}
For integers $1\le \ell<k$, let $\Pi^k_\ell$ denote the single-forbidden $\ell$-degree Tur\'an spectrum of $k$-uniform hypergraphs.  We introduce \emph{transfer functions} for this spectrum: explicit functions $f$ such that, for every $F$, there is another single $k$-graph $F^*$ with
$\pi_\ell(F^*)=f(\pi_\ell(F))$.  This gives a mechanism for producing new single-forbidden densities while retaining full control of the resulting value.

Our transfer functions are realized by a new family of suspension-type operations, called \emph{tree suspensions}.  From these operations we obtain three explicit maps: one acting on $\Pi^k_\ell$ for every $1\le\ell<k$, a second acting when $\ell\ge k/2$, and a third acting in the ordinary Tur\'an case $\ell=1$.  The common feature is a robust tree structure which gives the lower bound by a two-part construction and, in the regimes above, admits a matching embedding or Lagrangian upper bound.

As a first application, the universal transfer function propagates accumulation points.  Using the recent zero-accumulation results for $\ell\ge2$ together with the ordinary Tur\'an accumulation result of Conlon and Sch\"ulke, we prove that $\Pi^k_\ell$ has infinitely many accumulation points for every $k\ge3$ and every $1\le\ell<k$.  This recovers, in particular, the known infinitude of accumulation points in the ordinary and codegree spectra.

As a second application, combining two independent transfer functions forces algebraic degrees to grow.  For every $k\ge3$ and every $\ell\in\{1,\lceil k/2\rceil,\ldots,k-2\}$, the spectrum $\Pi^k_\ell$ contains algebraic numbers of arbitrarily large degree over $\mathbb Q$.  Thus the arithmetic complexity previously known for finite forbidden families already occurs in the single-forbidden spectrum, both for ordinary Tur\'an density and for a broad range of degree Tur\'an densities.
\end{abstract}



\section{Introduction}

A central problem in extremal combinatorics is to determine the \emph{Tur\'an number} $\ex(n,\mathcal F)$, the maximum number of edges in an $n$-vertex $k$-uniform hypergraph, or $k$-graph, containing no member of a family $\mathcal F$ of $k$-graphs.  The corresponding asymptotic parameter is the
\emph{Tur\'an density}
\[
        \pi(\mathcal F)=\lim_{n\to\infty}\frac{\ex(n,\mathcal F)}{\binom n k},
\]
where the limit exists by the classical averaging argument of Katona, Nemetz and Simonovits \cite{KatonaNemetzSimonovits}. 
When $\mathcal F=\{F\}$ consists of a single $k$-graph $F$, we write simply $\pi(F)$ for $\pi(\mathcal F)$.

For graphs, the Erd\H{o}s--Stone--Simonovits theorem~\cite{ErdosSimonovits,ErdosStone} essentially determines all Tur\'an densities. 
However, for hypergraphs of uniformity at least three, the situation is radically different.  Even Tur\'an's original problem of determining $\pi(K_4^{(3)})$ remains open, where $K_4^{(3)}$ denotes the complete $3$-graph on four vertices. 
This scarcity of exact values has led to a complementary point of view: rather than asking only for individual Tur\'an densities, one may study the \emph{spectrum} of all possible Tur\'an densities and ask about its topological and arithmetic structure.
See, for instance, the surveys \cite{FurediSurvey,KeevashSurvey,SidorenkoSurvey}.

Let $\Pi^{k}$, $\Pi^k_{\mathrm{fin}}$ and $\Pi^k_{\infty}$ denote the single, finite-family and arbitrary-family Tur\'an spectrum for $k$-graphs:
\[
\Pi^{k}=\{\pi(F) : F\text{ is a single }k\text{-graph}\},
\]
\[
        \Pi^k_{\mathrm{fin}}
        =\{\pi(\mathcal F) : \mathcal{F}\text{ is a finite family of } k \text{-graphs}\},
\]
and
\[
        \Pi^k_{\infty}
        =\{\pi(\mathcal F) : \mathcal{F}\text{ is an arbitrary family of } k \text{-graphs}\}.
\]
Clearly $\Pi^{k}\subseteq\Pi^k_{\mathrm{fin}}\subseteq\Pi^k_{\infty}$.  For $k=2$ these sets are determined by the Erd\H{o}s--Stone--Simonovits theorem.  For $k\ge3$, even very coarse structural questions about these spectra are difficult.

The first such structural questions concerned gaps and accumulation points.  Erd\H{o}s~\cite{Erdos1964} proved that $\Pi^k_{\infty}\cap(0,k!/k^k)=\varnothing$ and later conjectured that every point of $[0,1)$ has a non-trivial gap immediately to its right in $\Pi^k_\infty$~\cite{Erdos1983Conjectures}.  Frankl and R\"odl~\cite{FranklRodl} disproved this jumping conjecture, and several refinements followed; see, for example,~\cite{BaberTalbot,FranklPengRodlTalbot}.  In the single Tur\'an spectrum the problem is much more rigid.  Conlon and Sch\"ulke first proved that $\Pi^k$ has an accumulation point in $[0,1)$ for every $k\ge3$, with $1/2$ an accumulation point of $\Pi^3$, and then proved that $\Pi^k$ has infinitely many accumulation points for every $k\ge3$~\cite{ConlonSchulkeAccumulate,ConlonSchulkeInfinite}.

A second line of work asks how arithmetically complicated Tur\'an densities can be.  Baber and Talbot~\cite{BaberTalbot2012} found irrational elements of $\Pi^3_{\mathrm{fin}}$, disproving a conjecture of Chung and Graham.  Pikhurko~\cite{Pikhurko} developed a systematic theory for family spectra, showing in particular that $\Pi^k_{\mathrm{fin}}$ contains irrational numbers for every $k\ge3$ and that $\Pi^k_\infty$ has the cardinality of the continuum.  Liu and Pikhurko~\cite{LiuPikhurko}, answering a question of Grosu~\cite{Grosu}, proved that $\Pi^k_{\mathrm{fin}}$ contains algebraic numbers of arbitrarily large degree for every $k\ge3$. Very recently, Liu and Pikhurko~\cite{LiuPikhurkoIntervals} further proved that $\Pi^k_\infty$ contains non-degenerate intervals for every $k\ge3$, and in particular has full Hausdorff dimension. In contrast, irrational examples in the single Tur\'an spectrum are far more delicate; known constructions include work of Yan and Peng~\cite{YanPeng}, Wu~\cite{Wu}, and Kam\v{c}ev, Letzter and Pokrovskiy~\cite{KamcevLetzterPokrovskiy}.

These results reveal the richness of hypergraph Tur\'an spectra, but they also highlight the difficulty of such problems, especially in the single-forbidden setting.  The aim of this paper is to give a systematic way to build new single-forbidden densities from old ones.  We do this by introducing transfer functions acting directly on the more general single degree Tur\'an spectra and by realizing these functions through explicit hypergraph operations.

\subsection{Degree Tur\'an spectra and transfer functions}

We work in the more general setting of degree Tur\'an densities.  Let $1\le\ell<k$.  For a $k$-graph $H$ and an $\ell$-set $S\subseteq V(H)$, write $d_H(S)$ for the number of edges of $H$ containing $S$, and let $\delta_\ell(H)$ be the minimum of $d_H(S)$ over all such $S$.  The \textit{$\ell$-degree Tur\'an number} $\ex_\ell(n,F)$ is the largest possible value of $\delta_\ell(H)$ over all $n$-vertex $F$-free $k$-graphs $H$, and
\[
        \pi_\ell(F)=
        \lim_{n\to\infty}
        \frac{\ex_\ell(n,F)}{\binom{n-\ell}{k-\ell}}
\]
is the \textit{$\ell$-degree Tur\'an density}, whose existence was proved by Lo and Markstr\"om~\cite{LoMarkstrom}.  The case $\ell=1$ coincides with the ordinary Tur\'an density by the standard deletion argument, while $\ell=k-1$ is the codegree Tur\'an density.  We write $\Pi^k_\ell$, $\Pi^k_{\ell,\mathrm{fin}}$ and $\Pi^k_{\ell,\infty}$ for the corresponding single, finite-family and arbitrary-family degree Tur\'an spectrum.

Degree Tur\'an spectra display new behavior that is invisible in the ordinary case.  Whereas ordinary Tur\'an densities have the initial Erd\H{o}s gap, $\Pi^k_{\ell,\infty}$ is dense in $[0,1)$ for every $\ell\ge2$~\cite{LoMarkstrom,mubayi2007co}.  In the codegree case, Gao, Pikhurko, Rong and Sun~\cite{GaoPikhurkoRongSun} recently proved that $\Pi^k_{k-1,\mathrm{fin}}$ contains all rational numbers for $k\ge3$.  For single degree Tur\'an spectra, Keevash and Zhao~\cite{keevash2007codegree} proved that $\Pi^2\subseteq\Pi^k_{k-1}$, and every $\alpha\in\Pi^2$ is now known to be an accumulation point of $\Pi^k_{k-1}$ for $k\ge3$~\cite{Accu-coTuran-Bjarne,PigaSchulke}.  The authors~\cite{DingLiuYang} proved that $0$ is an accumulation point of $\Pi^k_\ell$ for all $\ell\ge2$.  These seed accumulation results are the starting point for our transfer method below.

A function $f$ is called a \emph{transfer function} of $\Pi^k_\ell$ if $f(\beta)\in\Pi^k_\ell$ for every $\beta\in\Pi^k_\ell$.  The key point is that such a function must be realized in the single-forbidden category: starting from one forbidden $k$-graph $F$, we must construct another single $k$-graph $F'$ with $\pi_\ell(F')=f(\pi_\ell(F))$.  Our main structural theorem gives three such functions.  Set
\[
\Phi_{k,\ell}(x)=1-\left(1+(1-x)^{-1/(k-\ell)}\right)^{-(k-\ell)},
\qquad
\Psi(x)=\frac{1}{2-x},
\]
and
\[
\Upsilon_k(x)=\left(\frac{k-1}{k-x}\right)^{k-1}.
\]

\begin{theorem}\label{tree-suspension}
Let $1\le \ell<k$. Then $\Phi_{k,\ell}(x)$ is a transfer function of $\Pi^k_{\ell}$. Moreover, $\Psi(x)$ is a transfer function of $\Pi^k_{\ell}$ whenever $\frac{k}{2}\leq\ell<k$, and $\Upsilon_k(x)$ is a transfer function of $\Pi^k_{1}=\Pi^k$. 
\end{theorem}

Note that $\Phi_{k,k-1}=\Psi=\Upsilon_2$.  Thus the codegree case receives only one transfer function from this construction.  By contrast, when $k\ge3$ and $\ell\in\{1,\lceil k/2\rceil,\ldots,k-2\}$, we obtain two distinct transfer functions acting on the same spectrum; this is what drives the algebraic-degree result below.

\subsection{Accumulation points}

Let $\operatorname{Acc}(\Pi^k_\ell)\subseteq[0,1]$ be the set of accumulation points of $\Pi^k_\ell$.  Since $\Phi_{k,\ell}$ is continuous on $[0,1)$, \cref{tree-suspension} immediately yields the following result, which in particular recovers recent results of Conlon and Sch\"ulke \cite{ConlonSchulkeInfinite} and of Li, Liu, Sch\"ulke and Sun \cite{Accu-coTuran-Bjarne} on the infinitude of $\operatorname{Acc}(\Pi^k_1)$ and $\operatorname{Acc}(\Pi^k_{k-1})$, respectively.

\begin{theorem}\label{accumulation-orbits}
Let $1\le\ell<k$, and suppose that
$\beta\in\operatorname{Acc}(\Pi^k_\ell)\cap[0,1)$. Then $\Phi_{k,\ell}(\beta) \in\operatorname{Acc}(\Pi^k_\ell)$.
Consequently, $\operatorname{Acc}(\Pi^k_\ell)$ is infinite for $k\geq3$.
\end{theorem}
\begin{proof}
Choose distinct $\beta_m\in\Pi^k_\ell$ with $\beta_m\to\beta$.  By \cref{tree-suspension}, $\Phi_{k,\ell}(\beta_m)\in\Pi^k_\ell$ for every $m$.  Since $\Phi_{k,\ell}$ is continuous and strictly increasing on $[0,1)$, the values $\Phi_{k,\ell}(\beta_m)$ are distinct and converge to $\Phi_{k,\ell}(\beta)$.  Hence $\Phi_{k,\ell}(\beta)$ is an accumulation point.

For $\ell\ge2$, it is known that $0\in\operatorname{Acc}(\Pi^k_\ell)$~\cite{DingLiuYang}; for $\ell=1$, it is known that $\operatorname{Acc}(\Pi^k_1)\cap[0,1)$ is non-empty when $k\ge3$~\cite{ConlonSchulkeAccumulate}.  Since $\Phi_{k,\ell}(x)>x$ for all $x\in[0,1)$, iterating $\Phi_{k,\ell}$ gives infinitely many accumulation points.
\end{proof}

The transfer function $\Phi_{k,\ell}(x)$ does more than propagate accumulation points:  it gives explicit infinite sequences. 
Let  $u_\ell(x)=(1-x)^{-1/(k-\ell)}$. Then $u_\ell(\Phi_{k,\ell}(x))=u_\ell(x)+1$.
Hence the $q$-fold iterate is
\[
        \Phi_{k,\ell}^{q}(x)=1-\left(q+(1-x)^{-1/(k-\ell)}\right)^{-(k-\ell)}.
\]
Thus, whenever $\beta\in\Pi_\ell^k$, all numbers 
\[
1-\left(q+(1-\beta)^{-1/(k-\ell)}\right)^{-(k-\ell)}\qquad(q\ge0)
\]
belong to $\Pi_\ell^k$.
Consequently, for $\ell\geq 2$ 
\[
1-\frac{1}{(q+1)^{k-\ell}}
\in\operatorname{Acc}(\Pi^k_\ell)
\qquad(q\ge0).
\]
Moreover, for the ordinary $3$-graph spectrum, Conlon and Sch\"ulke \cite{ConlonSchulkeAccumulate} proved that $1/2\in\operatorname{Acc}(\Pi^3)$. Applying $\Phi_{3,1}$ to this seed produces infinitely many irrational accumulation points of $\Pi^3$.

\subsection{Algebraic complexity}

The transfer viewpoint also gives arithmetic information.  The map $\Phi_{k,\ell}$ is linear after the change of variable $u_\ell=(1-x)^{-1/(k-\ell)}$, so by itself it does not force algebraic degrees to grow indefinitely.  The growth comes from combining this translation with a second transfer function: $\Psi$ in the range $k/2\le\ell\le k-2$, and $\Upsilon_k$ in the ordinary case $\ell=1$.  Starting from the zero-density seed given by a single $k$-edge, the alternating radical extensions have full degree at each step.

\begin{theorem}\label{thm:degree-high-degree}
Let $k\geq 3$ and $\ell\in\{1,\lceil k/2\rceil,\lceil k/2\rceil+1,\dots,k-2\}$.  For every
integer $D\ge1$, there exists a single $k$-graph $F$ such that $\pi_\ell(F)$ is algebraic of degree at least $D$ over $\mathbb Q$.
\end{theorem}

Thus the arithmetic complexity theorem of Liu and Pikhurko for finite forbidden families already has a single-forbidden analogue in the ordinary Tur\'an spectrum and in a broad range of degree Tur\'an spectra.  Combining this theorem with the accumulation-point transfer also gives infinitely many algebraic accumulation points of arbitrarily large degree in the same parameter ranges.

\subsection{Proof sketch}\label{sec:pf-idea}

We outline the construction behind \cref{tree-suspension}.  Given a non-empty $k$-graph $F$ with $\pi_\ell(F)=\beta$, we build a new single forbidden hypergraph whose $\ell$-degree Tur\'an density is an explicit function of $\beta$.  The construction is a family of \emph{tree suspensions}.

Let $s=k-\ell$ and $0\le t\le s-1$.  The $(\ell,t)$-tree suspension $\mathcal T^{(k)}_{\ell,t}(F)$ starts with $\ell$ disjoint root copies of $F$.  For each crossing $\ell$-tuple of root vertices, we attach a choice tree of depth $s-t$ whose branching copies are copies of $F$; each terminal choice is then enlarged by $t$ new vertices to form a connecting $k$-edge.  The key feature of $\mathcal T^{(k)}_{\ell,t}(F)$ is a robustness property: deleting a set which takes fewer than $k-t$ vertices from every edge cannot destroy all copies of $F$.

This robustness gives the lower bound.  Split the host vertex set as $A\cup B$, put an $F$-free near-extremal $k$-graph on $A$, and add every crossing $k$-set that meets $A$ in at least $t+1$ vertices.  If a copy of the tree suspension existed, deleting its vertices in $B$ would leave a copy of $F$ inside $A$, contradicting the choice of the graph on $A$.  Optimizing $|A|/n$ gives the candidate transfer value.

The upper bounds are the main work.  For $t=0$, a recursive embedding lemma finds the required choice trees inside dense links, giving the universal transfer function $\Phi_{k,\ell}$.  For $t=s-1$ and $\ell\ge k/2$, the same philosophy reduces to finding a copy of $F$ in a large support set of a link, giving the transfer function $\Psi$.  When $\ell=1$, the corresponding embedding argument is no longer strong enough; instead, we use the Lagrangian method together with Sidorenko's extension theorem to prove the ordinary transfer function $\Upsilon_k$.  These three sharp cases are precisely the transfer functions used for the accumulation and algebraic-complexity results.

\section{Tree suspensions: Proof of \cref{tree-suspension}}\label{sec:robust-tree}
Throughout this section, we fix integers $1\le \ell<k$ and write $s=k-\ell$. As outlined in \cref{sec:pf-idea}, our transfer functions arise from a unified tree-suspension construction applied to an arbitrary non-empty $k$-graph. We therefore study how the $\ell$-degree Tur\'an density changes under these tree suspensions. The three transfer functions in \cref{tree-suspension} will then follow immediately from Theorems~\ref{all-degree}, \ref{support-density} and \ref{turan-density}.

\subsection{The construction of $\mathcal T^{(k)}_{\ell,t}(F)$}

In this subsection, we give the construction of the $(\ell,t)$-tree suspension $\mathcal T^{(k)}_{\ell,t}(F)$ built recursively from a $k$-graph $F$, and establish a key robust property of it. 

Let $1\leq j<k$ be integers. A \textit{$(k,j)$-pattern} is a triple $P=(V(P),E_k(P),E_j(P))$ with $E_k(P)\subseteq V(P)^{(k)}$ and $E_j(P)\subseteq V(P)^{(j)}$, where the elements in $E_k(P)$ and $E_j(P)$ are respectively called \textit{$k$-edges} and \textit{$j$-edges} of $P$. 
We employ the following particular $(k,j)$-patterns for the construction of $(\ell,t)$-tree suspensions.  

\begin{defn}[Choice-tree patterns]\label{con:choice-trees}
Let $F$ be a non-empty $k$-graph with vertex set $[f]$. We recursively define a \textit{choice-tree $(k,j)$-pattern} $T_j(F)$ for each $1\le j\le s$.

For $j=1$, take a copy of $F$ with vertex set $X=\{x_1,\ldots,x_f\}$. The $k$-edges of $T_1(F)$ are the edges of this copy of $F$, and the $1$-edges are all singletons $\{x_a\}$ {\rm ($a\in[f]$)}.

For $j>1$, first take a root copy of $F$ with vertex set $X=\{x_1,\ldots,x_f\}$. 
For every $a\in[f]$, attach a disjoint copy $T^{a}_{j-1}(F)$ of $T_{j-1}(F)$. 
The $k$-edges of $T_j(F)$ are the $k$-edges of the root copy of $F$ together with all $k$-edges in the attached copies of $T_{j-1}(F)$. 
The $j$-edges are all sets $\{x_a\}\cup K$, where $a\in[f]$ and $K\in E_{j-1}(T^a_{j-1}(F))$.
\end{defn}

Clearly, $T_j(F)$ contains at least one copy of $F$. The next lemma establishes that this property is robust against the removal of any independent set of the underlying $j$-graph of $T_j(F)$.

\begin{lemma}\label{lem:tree-robustness}
Let $1\le j\le s$. For every subset $S\subseteq V(T_j(F))$, either the induced $k$-graph of $T_j(F)-S$ contains a copy of $F$, or $S$ spans at least one $j$-edge of $T_j(F)$.
\end{lemma}

\begin{proof}
We argue by induction on $j$. For $j=1$, if $S$ does not span a $1$-edge of $T_1(F)$, then $S=\varnothing$. Hence, the underlying $k$-graph of $T_1(F)-S$ is a copy of $F$.

Assume now that $j>1$ and that the claim is known for $j-1$. Let $X=\{x_1,\ldots,x_f\}$ be the vertex set of the root copy of $F$ in $T_j(F)$. If $S$ is disjoint from $X$, then $X$ is contained in the complement of $S$ and we are done. 
Otherwise, choose $x_a\in S\cap X$. Applying the induction hypothesis to the attached copy $T^a_{j-1}(F)$, we conclude that either the induced $k$-graph of $T^a_{j-1}(F)-S$ contains a copy of $F$ (and so does $T_j(F)-S$), or $S\setminus\{x_a\}$ spans a $(j-1)$-edge $K\in E_{j-1}(T^a_{j-1}(F))$, which implies that $\{x_a\}\cup K$ is a $j$-edge of $T_j(F)$ spanned by $S$.
\end{proof}

We now present the definition of $(\ell,t)$-tree suspensions.

\begin{defn}[Tree suspensions]\label{con:robust-tree-suspension}
Let $0\leq t\leq s-1$ and let $F$ be a non-empty $k$-graph with vertex set $[f]$. 
The \textit{$(\ell,t)$-tree suspension} $\mathcal T^{(k)}_{\ell,t}(F)$ of $F$ is defined as follows.

First, take $\ell$ pairwise disjoint root copies of $F$ with vertex sets $X^1,\ldots,X^\ell$ respectively, where $X^a=\{x^a_1,\ldots,x^a_f\}$ for $1\leq a\leq \ell$, and
include all edges of these copies of $F$ in $\mathcal T^{(k)}_{\ell,t}(F)$.

Next, for every ordered $\ell$-tuple $\mathbf a=(a_1,\ldots,a_\ell)\in[f]^\ell$, take a disjoint copy $T^{\mathbf a}_{s-t}(F)$ of the choice-tree pattern $T_{s-t}(F)$, and include all $k$-edges of this pattern in $\mathcal T^{(k)}_{\ell,t}(F)$.

Finally, for each $\mathbf a=(a_1,\ldots,a_\ell)\in[f]^\ell$ and each $K\in E_{s-t}(T^{\mathbf a}_{s-t}(F))$, take a new $t$-set $Y^{\mathbf a}_K$ and include the connecting $k$-edge
$e^{\mathbf a}(K)=\{x^1_{a_1},\ldots,x^\ell_{a_\ell}\}\cup K\cup Y^{\mathbf a}_K$ in $\mathcal T^{(k)}_{\ell,t}(F)$.
\end{defn}

\Cref{fig:R31-suspension} illustrates this construction in its smallest genuinely non-codegree instance, the $(1,0)$-tree suspension $\mathcal T^{(3)}_{1,0}(F)$ for $F$ a single edge (so $k=3$, $\ell=1$, $s=2$).


Similar to choice-tree patterns, the $(\ell,t)$-tree suspension $\mathcal T^{(k)}_{\ell,t}(F)$ also enjoys robustness with respect to the property of containing a copy of $F$ upon deletion of any set intersecting every edge in fewer than $k-t$ vertices.

\begin{lemma}\label{lem:independent-robustness}
Let $S\subseteq V(\mathcal T^{(k)}_{\ell,t}(F))$ be any set containing fewer than $k-t$ vertices of every edge of $\mathcal T^{(k)}_{\ell,t}(F)$. Then $\mathcal T^{(k)}_{\ell,t}(F)-S$ contains a copy of $F$.
\end{lemma}

\begin{proof}
If $S$ is disjoint from $X^a$ for some $1\leq a\leq \ell$, then the copy of $F$ induced on $X^a$ lies in $\mathcal T^{(k)}_{\ell,t}(F)-S$. 
Suppose instead that $S$ meets every $X^a$. 
Choose $x^a_{b_a}\in S\cap X^a$ for $1\leq a\leq \ell$,
and write $\mathbf b=(b_1,\ldots,b_\ell)$. Consider the attached choice-tree pattern $T^{\mathbf b}_{s-t}(F)$. By \cref{lem:tree-robustness}, either the induced $k$-graph of $T^{\mathbf b}_{s-t}(F)-S$ contains a copy of $F$, in which case $\mathcal T^{(k)}_{\ell,t}(F)-S$ also contains $F$, 
or $S\setminus\{x^1_{b_1},\ldots,x^\ell_{b_\ell}\}$ spans an $(s-t)$-edge $K$ of $T^{\mathbf b}_{s-t}(F)$. In the latter case
$\{x^1_{b_1},\ldots,x^\ell_{b_\ell}\}\cup K$
is a $(k-t)$-subset of the connecting $k$-edge $e^{\mathbf b}(K)$ of $\mathcal T^{(k)}_{\ell,t}(F)$, contradicting that $|S\cap e^{\mathbf b}(K)|<k-t$. Hence $\mathcal T^{(k)}_{\ell,t}(F)-S$ must contain a copy of $F$.
\end{proof}

\begin{figure}[ht]
\centering
\tikzset{
  Rvmain/.style={circle, draw=black, fill=white, inner sep=0pt, minimum size=4.4pt, line width=0.5pt},
  Rvhi/.style={circle, draw=black, fill=black, inner sep=0pt, minimum size=6.2pt, line width=0.5pt},
  Rfcopy/.style={draw=black, line width=0.8pt, rounded corners=6pt, fill=white},
  Rtwoedge/.style={dashed, draw=black!42, line width=0.5pt},
  Rkedge/.style={densely dashed, draw=black, line width=1.4pt},
  Rcleg/.style={draw=black, line width=1.15pt},
  Rcleglt/.style={draw=black!50, line width=0.7pt},
  Rsubbox/.style={draw=black!60, line width=0.6pt, rounded corners=4pt},
  Rblab/.style={font=\footnotesize, text=black!75},
  Rvlab/.style={font=\scriptsize, text=black},
}
\resizebox{\textwidth}{!}{%
\begin{tikzpicture}[x=1cm,y=1cm]

  \foreach \a/\cx in {1/0, 2/6, 3/12}{
    \foreach \i in {1,2,3}{
      \pgfmathsetmacro{\yx}{\cx + (\i-2)*0.5}
      \coordinate (Y\a\i) at (\yx,3);
    }
    \foreach \c in {1,2,3}{
      \pgfmathsetmacro{\lcx}{\cx + (\c-2)*1.7}
      \foreach \b in {1,2,3}{
        \pgfmathsetmacro{\zx}{\lcx + (\b-2)*0.30}
        \coordinate (Z\a\c\b) at (\zx,0);
      }
    }
  }
  \coordinate (X1) at (5.45,6.2);
  \coordinate (X2) at (6.00,6.2);
  \coordinate (X3) at (6.55,6.2);
  \coordinate (H2) at (6.0,4.70);
  \coordinate (H3) at (11.4,4.60);
  \coordinate (Hh) at (1.70,4.00);

  \foreach \a/\cx in {1/0, 2/6, 3/12}{
    \draw[Rsubbox] (\cx-2.30,-0.55) rectangle (\cx+2.30,3.55);
    \node[Rblab, anchor=north] at (\cx,-0.72) {$T_2^{\a}(F)$};
  }

  \begin{scope}[on background layer]
    \draw[Rfcopy] (5.02,5.88) rectangle (6.98,6.52); 
    \foreach \a/\cx in {1/0, 2/6, 3/12}{
      \draw[Rfcopy] (\cx-0.86,2.66) rectangle (\cx+0.86,3.34); 
      \foreach \c in {1,2,3}{
        \pgfmathsetmacro{\lcx}{\cx + (\c-2)*1.7}
        \draw[Rfcopy] (\lcx-0.58,-0.34) rectangle (\lcx+0.58,0.34); 
      }
    }
  \end{scope}

  \foreach \a in {1,2,3}{
    \foreach \c in {1,2,3}{
      \foreach \b in {1,2,3}{ \draw[Rtwoedge] (Y\a\c) -- (Z\a\c\b); }
    }
  }

  \draw[Rcleglt] (H2) -- (X2);
  \draw[Rcleglt] (H2) -- (Y22);
  \draw[Rcleglt] (H2) -- (Z222);
  \draw[Rcleglt] (H3) -- (X3);
  \draw[Rcleglt] (H3) -- (Y32);
  \draw[Rcleglt] (H3) -- (Z322);

  \draw[Rcleg] (Hh) -- (X1);
  \draw[Rcleg] (Hh) -- (Y13);
  \draw[Rcleg] (Hh) -- (Z133);
  \draw[Rkedge] (Y13) -- (Z133); 
  \node[font=\scriptsize, anchor=west] at (1.92,4.18) {$\{x^1_a\}\cup K$};
  \node[font=\scriptsize, anchor=east] at (0.80,1.52) {$K$};

  \foreach \a in {1,2,3}{
    \foreach \i in {1,2,3}{ \node[Rvmain] at (Y\a\i) {}; }
    \foreach \c in {1,2,3}{ \foreach \b in {1,2,3}{ \node[Rvmain] at (Z\a\c\b) {}; } }
  }
  \node[Rvmain] at (X1) {}; \node[Rvmain] at (X2) {}; \node[Rvmain] at (X3) {};
  \node[Rvhi] at (X1) {}; \node[Rvhi] at (Y13) {}; \node[Rvhi] at (Z133) {};

  \node[Rblab, anchor=south] at (6.0,7.02) {root copy $X^1$};
  \node[Rvlab, anchor=south] at (5.45,6.60) {$x^1_1$};
  \node[Rvlab, anchor=south] at (6.00,6.60) {$x^1_2$};
  \node[Rvlab, anchor=south] at (6.55,6.60) {$x^1_3$};
  \node[Rvlab, anchor=south] at (0.40,3.48) {$y_c$};
  \node[Rvlab, anchor=west]  at (2.40,0.02) {$z_i$};
  \node[Rblab, anchor=west]  at (14.45,3.0) {\,root copy of $F$};
  \node[Rblab, anchor=west]  at (14.45,0.0) {\,leaf copies $T_1(F)$};

\end{tikzpicture}%
}
\caption{The $(1,0)$-tree suspension $\mathcal T^{(3)}_{1,0}(F)$ for $F$ a single edge. As $\ell=1$ there is a single root copy $X^1=\{x^1_1,x^1_2,x^1_3\}$ (top), so the construction is a rooted tree. For each $a\in[3]$ an attached choice-tree pattern $T^a_2(F)$ hangs below; it consists of a root copy of $F$ (middle, with vertices written $y_1,\dots,y_f$) each of whose vertices carries a leaf copy $T_1(F)$ (bottom, vertices $z_1,\dots,z_f$). The light dashed segments are the $2$-edges of the choice tree, $K=\{y_c,z_i\}\in E_2(T^a_2(F))$. Since $t=0$ the auxiliary set $Y^{\mathbf a}_K$ is empty, so a connecting edge is $e^{(a)}(K)=\{x^1_a\}\cup K$; the bold three-way fork shows one such edge for $a=1$, joining $x^1_a$, $y_c$ and $z_i$, while the light forks indicate the corresponding edges for $a=2,3$. Filled vertices are those of the displayed connecting edge. Each $T^a_2(F)$ in fact receives one connecting edge for every $K\in E_2(T^a_2(F))$; only one per subtree is drawn.}
\label{fig:R31-suspension}
\end{figure}
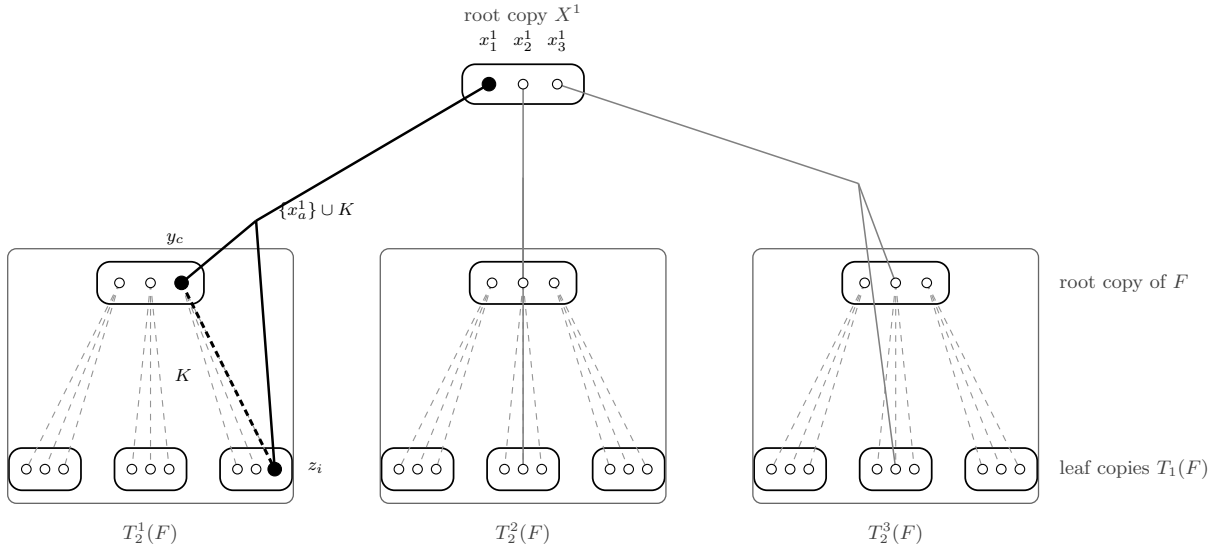


\subsection{The two-part lower bounds}\label{sec:two-part}

Let $F$ be a $k$-graph with $\pi_{\ell}(F)=\beta\in[0,1)$. In this subsection, we
provide a lower bound for $\pi_{\ell}(\mathcal T^{(k)}_{\ell,t}(F))$ using the robust property given in \cref{lem:independent-robustness}.

For every integer  $0\leq j\leq s$ and every real $x\in[0,1]$, let 
\[Q_{s,j}(x)=\sum_{i=j}^s\binom{s}{i}x^i(1-x)^{s-i},\]
which is the density of $s$-sets having at least $j$ vertices in a part of relative size $x$.
We use the convention that $\binom mi=0$ when $i<0$ or $i>m$, and set $Q_{s,j}(x)=1$ for $j<0$. 
Let $\gamma\in(0,1)$ be the unique solution of $Q_{s,t+1}(x)=Q_{s,t+1-\ell}(x)-(1-\beta)x^s$. 
Then 
\[\gamma=
\begin{cases}
(1+(1-\beta)^{1/s})^{-1},  &t=0,\\    
(2-\beta)^{-1/s}, & t+1=s\leq \ell, \\
\frac{k-1}{k-\beta}, & t+1=s\text{~and~}\ell=1.
\end{cases}\]
This further gives that 
\[Q_{s,t+1}(\gamma)=
\begin{cases}
1-\left(1+(1-\beta)^{-1/s}\right)^{-s},  &t=0,\\    
\frac{1}{2-\beta}, & t+1=s\leq \ell, \\
\left(\frac{k-1}{k-\beta}\right)^{k-1}, & t+1=s\text{~and~}\ell=1,
\end{cases}\]
which are the densities produced by the transfer functions in \cref{tree-suspension} with the input $\beta$. 

We first show that $Q_{s,t+1}(\gamma)$ is a lower bound of $\pi_{\ell}(\mathcal T^{(k)}_{\ell,t}(F))$.

\begin{lemma}\label{lower-bound}
Let $F$ be a non-empty $k$-graph with $\pi_{\ell}(F)=\beta$. Then 
$\pi_{\ell}(\mathcal T^{(k)}_{\ell,t}(F))\geq Q_{s,t+1}(\gamma)$.     
\end{lemma}
\begin{proof}
Choose $1/n\ll\eta\ll\varepsilon\ll\beta$ if $\beta>0$. For $\beta=0$, we choose $\eta=0$ and $1/n\ll\varepsilon\ll1$. To prove that $\pi_{\ell}(\mathcal T^{(k)}_{\ell,t}(F))\geq Q_{s,t+1}(\gamma)$, we next construct an $n$-vertex $\mathcal T^{(k)}_{\ell,t}(F)$-free $k$-graph $H$ on vertex set $A\cup B$ with minimum $\ell$-degree at least $(Q_{s,t+1}(\gamma)-\varepsilon)\binom{n}{s}$, where $|A|=\gamma n$ and $|B|=(1-\gamma)n$.

Since $\pi_{\ell}(F)=\beta$, there exists an $F$-free $k$-graph $H'$ on $\gamma n$ vertices with $\delta_{\ell}(H')\geq(\beta-\eta)\binom{\gamma n}{s}$. The $k$-graph $H$ is constructed by putting a copy of $H'$ on vertex set $A$ and, in addition, adding every $k$-set not contained in $A$ that intersects $A$ in at least $t+1$ vertices. 

We first prove that $H$ contains no copy of $\mathcal T^{(k)}_{\ell,t}(F)$. Assume for contradiction that $H$ admits a copy $\mathcal{T'}$ of $\mathcal T^{(k)}_{\ell,t}(F)$. Let $S=V(\mathcal{T'})\cap B$. By the construction of $H$, we have $|S\cap e|\leq k-t-1$ for every $e\in E(\mathcal{T'})$. 
It then follows from \cref{lem:independent-robustness} that $\mathcal{T'}-S$ contains a copy of $F$, which contradicts that $H[A]$ is $F$-free. 

It remains to verify the minimum $\ell$-degree condition.
Let $I$ be an $\ell$-subset of $V(H)$ and set $r=|I\cap A|$. If $r=\ell$, then 
\begin{align}
d_H(I)\geq & (\beta-\eta)\binom{|A|-\ell}{s}+\sum_{i=t+1-\ell}^{s-1}\binom{|A|-r}{i}\binom{|B|-\ell+r}{s-i}\nonumber\\
\geq & \sum_{i=t+1-\ell}^{s}\binom{|A|-r}{i}\binom{|B|-\ell+r}{s-i}-(1-\beta+\eta)\binom{|A|-\ell}{s}\nonumber\\
\geq& \left(Q_{s,t+1-\ell}(\gamma)-(1-\beta)\gamma^s-\varepsilon\right)\binom{n}{s}\nonumber\\
=&\left(Q_{s,t+1}(\gamma)-\varepsilon\right)\binom{n}{s},\nonumber
\end{align}
where the last equality follows from the choice of $\gamma$.
If $r<\ell$, then 
\begin{align}
d_H(I)\geq & \sum_{i=t+1-r}^{s}\binom{|A|-r}{i}\binom{|B|-\ell+r}{s-i}\nonumber\\
\geq & \left(Q_{s,t+1-r}(\gamma)-\varepsilon\right)\binom{n}{s}\nonumber\\
\geq& \left(Q_{s,t+1}(\gamma)-\varepsilon\right)\binom{n}{s}.\nonumber
\end{align}
Therefore, we conclude that $\delta_\ell(H)\geq (Q_{s,t+1}(\gamma)-\varepsilon)\binom{n}{s}$, finishing the proof.
\end{proof}


\subsection{$(\ell,0)$-tree suspensions}
Recall that when $t=0$
$$\gamma=\frac{1}{1+(1-\beta)^{1/s}}\text{~~~and~~~} Q_{s,t+1}(\gamma)=1-\left(1+(1-\beta)^{-1/s}\right)^{-s}=\Phi_{k,\ell}(\beta).$$
Therefore, it follows from \cref{lower-bound} that $\pi_{\ell}(\mathcal T^{(k)}_{\ell,0}(F))\geq\Phi_{k,\ell}(\beta)$ for every non-empty $k$-graph $F$ with $\pi_\ell(F)=\beta$.
We next prove that $\Phi_{k,\ell}(\beta)$ is also an upper bound of $\pi_\ell(\mathcal T^{(k)}_{\ell,0}(F))$, and thus $\Phi_{k,\ell}(x)$ is a transfer function of $\Pi^k_{\ell}$ for all $1\leq\ell<k$.

\begin{theorem}\label{all-degree}
Let $F$ be a non-empty $k$-graph with $\pi_\ell(F)=\beta$. Then $\pi_{\ell}(\mathcal T^{(k)}_{\ell,0}(F))=\Phi_{k,\ell}(\beta)$.    
\end{theorem}

We need two embedding lemmas for the proof of \cref{all-degree}. The first one enables us to find a copy of $F$ in a large vertex subset.

\begin{lemma}\label{lem:large-subsets-contain-F}
Let $F$ be a non-empty $k$-graph with $\pi_\ell(F)=\beta$. For every $\varepsilon>0$ and $\rho\in(0,1]$, the following holds for all sufficiently large $n$. Suppose that $H$ is an $n$-vertex $k$-graph with
$\delta_\ell(H)\ge (1-(1-\beta)\rho^s+\varepsilon)\binom{n-\ell}{s}$,
and $U\subseteq V(H)$ satisfies that $|U|\ge \rho n$, then $H[U]$ contains a copy of $F$.
\end{lemma}

\begin{proof}
Without loss of generality, we may assume that $|U|=\rho n$. Note that
\[
        \delta_\ell({H[U]})
        \ge \delta_\ell(H)-\left(\binom{n-\ell}{s}-\binom{|U|-\ell}{s}\right)\geq \left(\beta+\frac{\varepsilon}{2}\right)\binom{|U|-\ell}{s}
\]
for sufficiently large $n$.
Since $\pi_{\ell}(F)=\beta$,
there exists a copy of $F$ in $H[U]$.
\end{proof}

In the remainder of this subsection, let $F$ be a non-empty $k$-graph with $\pi_\ell(F)=\beta$. For simplicity, set 
\[
        \theta=(1-\Phi_{k,\ell}(\beta))^{1/s}
        =\frac{(1-\beta)^{1/s}}{1+(1-\beta)^{1/s}}
        \text{~~~and~~~}
        \rho_j=1-\theta^j\quad(1\le j\le s).
\]
Thus $\rho_s=\Phi_{k,\ell}(\beta)$ and $\rho_1=1-\theta=\gamma$.
Given a $k$-graph $H$ and a $j$-graph $J$ on a common vertex set, we say that $(H,J)$ contains a $(k,j)$-pattern $P$ if there is an injection $\psi:V(P)\to V(H)$ such that $\psi(e)\in E(H)$ and 
$\psi(e')\in E(J)$ for all $e\in E_k(P), e'\in E_j(P)$. 

The subsequent lemma provides a sufficient condition for the existence of the choice-tree pattern $T_j(F)$ in a $(k,j)$-pattern.

\begin{lemma}\label{lem:recursive-tree-embedding}
For every $\varepsilon>0$,
$0<\xi<\frac12\min\{\theta^{s},\varepsilon\}$
and $1\le j\le s$, the following holds for all sufficiently large $n$. 
Let $H$ be an $n$-vertex $k$-graph with
$\delta_\ell(H)\ge (\Phi_{k,\ell}(\beta)+\varepsilon)\binom{n-\ell}{s}$, and
let $J$ be a $j$-graph on $V(H)$ with
$|E(J)|\ge (\rho_j+\xi)\binom{n}{j}$.
Then $(H,J)$ contains the $(k,j)$-pattern $T_j(F)$.
\end{lemma}

\begin{proof}
We prove the statement by induction on $j$. 
For $j=1$, the $1$-graph $J$ is a vertex set of size at least $(\rho_1+\xi)n$. 
Recall that $\rho_1=\gamma$ and $\Phi_{k,\ell}(\beta)=Q_{s,1}(\gamma)=1-(1-\beta)\gamma^s$ (by the definition of $\gamma$).
Therefore, it follows from \cref{lem:large-subsets-contain-F} that $H[J]$ contains a copy of $F$, which gives an embedding of $T_1(F)$ into $(H,J)$ (since the $1$-edges of $T_1(F)$ are all its vertices).

Now let $j\ge2$ and assume the statement for $j-1$. We first embed the root copy of $F$ in $T_j(F)$.
Let 
$U=\left\{v\in V(J) : d_J(v)\ge (\rho_{j-1}+\xi)\binom{n-1}{j-1}\right\}.$
We claim that $|U|\geq \rho_1 n$.
On one hand,
$$\sum_{v\in V(H)}d_J(v)=j|E(J)|\ge j(\rho_j+\xi)\binom{n}{j}=(\rho_j+\xi)\,n\binom{n-1}{j-1}.$$
On the other hand,
$$\sum_{v\in V(H)}d_J(v)=\sum_{v\in U}d_J(v)+\sum_{v\notin U}d_J(v)
\le |U|\binom{n-1}{j-1}+(n-|U|)(\rho_{j-1}+\xi)\binom{n-1}{j-1}.$$
Therefore,

\[
        (\rho_j+\xi)\,n\binom{n-1}{j-1}
        \le |U|\binom{n-1}{j-1}+(n-|U|)(\rho_{j-1}+\xi)\binom{n-1}{j-1}.
\]
This yields that
\[
        |U|
        \ge \frac{\rho_j-\rho_{j-1}}{1-\rho_{j-1}-\xi}n
        =\frac{\theta^{j-1}\rho_1}{\theta^{j-1}-\xi}n
        \geq\rho_1 n,
\]
where the equality follows from that $\rho_j-\rho_{j-1}=\theta^{j-1}(1-\theta)=\theta^{j-1}\rho_1$ and $1-\rho_{j-1}=\theta^{j-1}$, and the last inequality follows from the choice of $\xi$. 
By \cref{lem:large-subsets-contain-F}, $H[U]$ contains a copy of $F$.  

Denote by $X=\{x_1,\ldots,x_f\}$ the vertex set of the copy of $F$ in $H[U]$.
We now embed the $f$ attached pairwise disjoint choice-tree patterns $T^a_{j-1}(F)$ greedily for $a=1,2,\ldots,f$ such that $\{x_a\}\cup K$ is a $j$-edge of $T_j(F)$ for every $(j-1)$-edge in the copy of $T^a_{j-1}(F)$. 
When embedding $T^a_{j-1}(F)$, the set of vertices already occupied (by the root copy of $F$ together with the copies of $T^1_{j-1}(F),\ldots,T^{a-1}_{j-1}(F)$) has size at most $f+f|V(T_{j-1}(F))|$. Let $H'$ be obtained from $H$ by deleting these used vertices, and let $J'_a$ be the $(j-1)$-graph obtained by restricting the link $L_J(x_a)$ to $V(H')$. Then for sufficiently large $n$,
$|E(J'_a)|\geq (\rho_{j-1}+\xi/2)\binom{n}{j-1}$ as $x_a\in U$ and $\delta_{\ell}(H')\geq(\Phi_{k,\ell}(\beta)+\varepsilon/2)\binom{n}{s}$.
Applying the induction hypothesis with margin $\xi/2$ in place of $\xi$ and $\varepsilon/2$ in place of $\varepsilon$, we find a copy of $T^a_{j-1}(F)$ in $(H',J'_a)$ which is disjoint from all previously used vertices. Since every $(j-1)$-edge $K$ of this copy is contained in $L_J(x_a)$, it follows that $\{x_a\}\cup K\in J$. 
After $f$ iterations we have all copies of the $f$ choice-tree patterns $T^1_{j-1}(F),\ldots,T^{f}_{j-1}(F)$, which together with the copy of $F$ in $H[U]$ yield an embedding of $T_j(F)$ into $(H,J)$.
\end{proof}

We are now ready to prove \cref{all-degree}. 

\begin{proof}[Proof of \cref{all-degree}]
It follows from \cref{lower-bound} that $\pi_{\ell}(\mathcal T^{(k)}_{\ell,0}(F))\geq\Phi_{k,\ell}(\beta)$.
We next prove that $\pi_{\ell}(\mathcal T^{(k)}_{\ell,0}(F))\leq\Phi_{k,\ell}(\beta)$.
Let $0<\varepsilon<1$ and $H$ be an $n$-vertex $k$-graph with
$\delta_\ell(H)\ge (\Phi_{k,\ell}(\beta)+\varepsilon)\binom{n-\ell}{s}$.
We next show that $H$ contains $\mathcal T^{(k)}_{\ell,0}(F)$ if $n$ is sufficiently large.

Note that $\Phi_{k,\ell}(\beta)>\beta=\pi_{\ell}(F)$. Therefore, we can find $\ell$ pairwise disjoint copies of $F$ in $H$ greedily if $n$ is sufficiently large. Let $X^1,\ldots,X^\ell$ be the vertex sets of these copies of $F$ with $X^a=\{x^a_1,\ldots,x^a_f\}$ for $1\leq a\leq \ell$.
For $\mathbf a=(a_1,\ldots,a_\ell)\in[f]^\ell$, let $X_{\mathbf a}=\{x^1_{a_1},\ldots,x^\ell_{a_\ell}\}$ and $L_{\mathbf a}=L_H(X_{\mathbf a})$. 
By the definition of $\mathcal T^{(k)}_{\ell,0}(F)$, to embed it into $H$, it now suffices to find disjoint copies of the $f^\ell$ 
choice-tree patterns $T^{\mathbf a}_s(F)$ in $(H-\cup_{a=1}^\ell X^a, L_{\mathbf a}-\cup_{a=1}^\ell X^a)$, as this ensures that $X_{\mathbf a}\cup K\in E(H)$ for every $s$-edge $K$ in the copy of $T^{\mathbf a}_s(F)$ and every $\mathbf a\in[f]^\ell$.

Now fix an $\ell$-tuple $\mathbf a\in[f]^\ell$ for which the required copy of $T^{\mathbf a}_s(F)$ is not determined. Recall that $\rho_s=\Phi_{k,\ell}(\beta)$.
Then the minimum $\ell$-degree assumption gives 
$|E(L_{\mathbf a})|=d_H(X_{\mathbf a})\ge (\rho_s+\varepsilon)\binom{n-\ell}{s}$.
Let $H'$ and $L'_{\mathbf a}$ be obtained from $H$ and $L_{\mathbf a}$ by deleting the set of vertices already used which has size at most $\ell f+f^\ell|V(T_s(F))|$. 
Then 
$$\delta_\ell(H')\geq (\Phi_{k,\ell}(\beta)+\varepsilon/2)\binom{n-\ell}{s}\text{~~~and~~~}|E(L'_{\mathbf a})|\ge (\rho_s+\varepsilon/2)\binom{n}{s}$$
for sufficiently large $n$. 
Choose $0<\xi<\min\{\varepsilon/4,\theta^s/2\}=\frac12\min\{\theta^s,\varepsilon/2\}$. Since $|E(L'_{\mathbf a})|\ge (\rho_s+\xi)\binom{n}{s}$ for large $n$, applying \cref{lem:recursive-tree-embedding} with $\varepsilon/2$ in place of $\varepsilon$ gives a copy of $T^{\mathbf a}_s(F)$ in $(H',L'_{\mathbf a})$. After $f^\ell$ iterations we have all the desired copies of choice-tree patterns.

Thus every sufficiently large $H$ with normalized minimum $\ell$-degree at least $\Phi_{k,\ell}(\beta)+\varepsilon$ contains $\mathcal T^{(k)}_{\ell,0}(F)$. Since $\varepsilon>0$ is arbitrary,
$\pi_\ell(\mathcal T^{(k)}_{\ell,0}(F))\le \Phi_{k,\ell}(\beta)$.
\end{proof}


\subsection{$(\ell,s-1)$-tree suspensions with $\ell\geq k/2$}

Recall that when $t+1=s\leq\ell$
$$\gamma=\frac{1}{(2-\beta)^{1/s}}\text{~~~and~~~} Q_{s,t+1}(\gamma)=\frac{1}{2-\beta}=1-(1-\beta)\gamma^s=\gamma^s=\Psi(\beta).$$
This subsection is devoted to proving the following theorem, which implies that $\Psi(x)$ is a transfer function of $\Pi_{\ell}^k$ whenever $\ell\geq k/2$.

\begin{theorem}\label{support-density}
Let $F$ be a non-empty $k$-graph with $\pi_\ell(F)=\beta$. Then  $\pi_{\ell}(\mathcal T^{(k)}_{\ell,s-1}(F))=\Psi(\beta)$ if $\ell\geq k/2$.  
\end{theorem}
\begin{proof}[Proof of \cref{support-density}]
Let $\ell\geq k/2$. It follows from \cref{lower-bound} that $\pi_{\ell}(\mathcal T^{(k)}_{\ell,s-1}(F))\geq\Psi(\beta)$.
We next prove that $\pi_{\ell}(\mathcal T^{(k)}_{\ell,s-1}(F))\leq\Psi(\beta)$.
Let $0<\varepsilon<1$ and $H$ be an $n$-vertex $k$-graph with
$\delta_\ell(H)\ge (\Psi(\beta)+\varepsilon)\binom{n-\ell}{s}$.
We next show that $H$ contains $\mathcal T^{(k)}_{\ell,s-1}(F)$ if $n$ is sufficiently large.

Note that $\Psi(\beta)>\beta=\pi_\ell(F)$. Therefore, we can find $\ell$ pairwise disjoint copies of $F$ in $H$ greedily if $n$ is sufficiently large. Let $X^1,\ldots,X^\ell$ be the vertex sets of these copies of $F$ with $X^a=\{x^a_1,\ldots,x^a_f\}$ for $1\leq a\leq \ell$. For $\mathbf a=(a_1,\ldots,a_\ell)\in[f]^\ell$, let $X_{\mathbf a}=\{x^1_{a_1},\ldots,x^\ell_{a_\ell}\}$ and $L_{\mathbf a}=L_H(X_{\mathbf a})$.

Let $F^+$ be the $(k,s)$-pattern such that the underlying $k$-graph is a copy of $F$ and the underlying $s$-graph is a matching of size $f$ with each edge containing exactly one vertex in the copy of $F$. 
Then by the definition of $\mathcal T^{(k)}_{\ell,s-1}(F)$, to embed it into $H$, it suffices to find a copy $F^+_{\mathbf a}$ of the $(k,s)$-pattern $F^+$ in $(H-\cup_{a=1}^\ell X^a, L_{\mathbf a}-\cup_{a=1}^\ell X^a)$ for each $\mathbf a\in[f]^\ell$ such that $F^+_{\mathbf x}$ and $F^+_{\mathbf y}$ are vertex disjoint for all distinct $\mathbf x, \mathbf y \in[f]^\ell$.

Now fix an $\ell$-tuple $\mathbf a\in[f]^\ell$ for which $F^+_{\mathbf a}$ is not determined.
Let $H'$ and $L'_{\mathbf a}$ be obtained from $H$ and $L_{\mathbf a}$ by deleting the set of vertices already used which has size at most $\ell f+f^\ell|V(F^+)|$. 
The minimum degree assumption gives
\[
|E(L'_{\mathbf a})|\ge 
\left(\Psi(\beta)+\frac{\varepsilon}{2}\right)\binom ns
= \left(\gamma^s+\frac{\varepsilon}{2}\right)
\binom ns
\]
for large $n$. 
Choose $\eta\ll\varepsilon$, and let 
\[
U=
\left\{v\in V(L'_{\mathbf a}) : d_{L'_{\mathbf a}}(v)\ge \eta\binom{n-1}{s-1}\right\}.
\]
Then we have 
\[
\binom{|U|}{s}+n\cdot\eta\binom{n-1}{s-1}
\geq \left(\gamma^s+\frac{\varepsilon}{2}\right)\binom ns,
\]
which yields that $|U|\geq \gamma n$.
Note that 
\[
\delta_\ell(H')\geq \left(\Psi(\beta)+\frac{\varepsilon}{2}\right)\binom ns
= \left(1-(1-\beta)\gamma^s+\frac{\varepsilon}{2}\right)
\binom ns.
\]
Therefore, it follows from \cref{lem:large-subsets-contain-F} that $H'[U]$ contains a copy of $F$.
Since every vertex in $U$ has degree at least $\eta\binom{n-1}{s-1}$ in $L'_{\mathbf a}$, we can greedily find $f$ pairwise disjoint edges in $L'_{\mathbf a}$, each containing exactly one prescribed vertex from the copy of $F$ in $H'[U]$ and avoiding the other vertices of that copy. 
Thus we find the desired $F^+_{\mathbf a}$ in $(H',L'_{\mathbf a})$.

After $f^\ell$ iterations we have all the desired copies of $F^+$, which together with the $\ell$ root copies of $F$ form a copy of $\mathcal T^{(k)}_{\ell,s-1}(F)$. This completes the proof.
\end{proof}

\subsection{$(1,k-2)$-tree suspensions}

In the preceding subsection, we established that $\pi_{\ell}(\mathcal T^{(k)}_{\ell,s-1}(F))=Q_{s,s}(\gamma)$ for $\ell\geq k/2$. A key ingredient of the proof was locating a copy of $F$ within a suitably sized vertex subset by virtue of the minimum 
$\ell$-degree condition (\cref{lem:large-subsets-contain-F}). 
For $\ell<k/2$, however, the value of $Q_{s,t+1}(\gamma)$ falls below $1-(1-\beta)\gamma^s$, which renders \cref{lem:large-subsets-contain-F} inapplicable. 

Adopting a substantially different approach that combines the Lagrangian method with a finite-family form of Sidorenko's extension theorem, we show that the equality $\pi_{\ell}(\mathcal T^{(k)}_{\ell,s-1}(F))=Q_{s,s}(\gamma)$ remains true for $\ell=1$. 
Recall that
$$Q_{s,s}(\gamma)=\left(\frac{k-1}{k-\beta}\right)^{k-1}=\Upsilon_k(\beta)$$
when $t+1=s$ and $\ell=1$.
Thus $\Upsilon_k(x)$ is a transfer function of $\Pi^k_{1}$.

\begin{theorem}\label{turan-density}
Let $F$ be a non-empty $k$-graph with $\pi(F)=\beta$. Then  $\pi(\mathcal {T}^{(k)}_{1,k-2}(F))=\Upsilon_k(\beta)$.      
\end{theorem}

To prove \cref{turan-density}, a central tool we employ is Sidorenko's extension theorem formulated for finite families of 
$k$-graphs. 
We begin by introducing standard definitions and notation.
Given a $k$-graph, its \textit{extension} $\operatorname{Ext}(F)$ is constructed by attaching a set $X_{uv}$ of $k-2$ new vertices to each uncovered pair $\{u,v\}$ (i.e., $d_F(\{u,v\})=0$), and then adding the edge $\{u,v\}\cup X_{uv}$.
The \textit{Lagrangian} of a $k$-graph $H$ is given by
\[
\lambda(H)=\max\left\{\sum_{e\in E(H)}\prod_{v\in e}x_v : x_v\ge0,\ \sum_{v\in V(H)}x_v=1\right\}.
\]
For a finite family $\mathcal{F}$ of $k$-graphs, let $\operatorname{Ext}(\mathcal{F})=\{\operatorname{Ext}(F) : F\in\mathcal{F}\}$ and $\pi_\lambda(\mathcal F)=\sup\{k!\lambda(H) : H\text{ is }\mathcal F\text{-free}\}$.

The result below is the finite-family form of Sidorenko's extension theorem, whose proof is the same as in the classical single-graph case; see, for example, \cite{sidorenko1987,SidorenkoExtension}. 

\begin{proposition}\label{lem:extension}
Let $\mathcal{F}$ be a finite family of $k$-graphs. Then $\pi(\operatorname{Ext}(\mathcal{F}))=\pi_\lambda(\mathcal F)$. 
\end{proposition}

 We now give the proof of \cref{turan-density}.

\begin{proof}[Proof of \cref{turan-density}]
Let $F$ be a non-empty $k$-graph on vertex set $[f]$ with $\pi(F)=\beta$.
The lower bound follows from \cref{lower-bound}.
To prove the upper bound, we consider the following auxiliary finite family of $k$-graphs.

Let $\mathcal M$ be the finite family of all $k$-graphs $M$
for which there are homomorphisms
\[
\phi:F\to M,
\qquad
\phi_i:F\to M-\phi(i)\quad (\forall i\in[f]),
\]
and $M$ is the union of the images of these homomorphisms.  
Clearly, such $k$-graphs $M$ exist. For instance, one can take $M$ to be the union of $f+1$ disjoint copies of $F$.

\begin{claim}\label{auxiliary}
For every $M\in\mathcal{M}$, there is a homomorphism from $\mathcal {T}^{(k)}_{1,k-2}(F)$ to $\operatorname{Ext}(M)$.
\end{claim}
\begin{proof}
We use the same symbols as in the definition of $(\ell,t)$-tree suspensions. Fix $M \in \mathcal{M}$ and homomorphisms $\phi, \phi_1, \dots, \phi_f$ as in the definition of $\mathcal{M}$. Then map the root copy of $F$ in $\mathcal {T}^{(k)}_{1,k-2}(F)$ by $\phi$, and map the attached copy $F_i$ by $\phi_i$ for each $i \in [f]$.

It remains to map the $k-2$ vertices in $Y_{v}^i$ for every $i\in[f]$ and $v\in V(F_i)$. Since $\phi_i(v) \neq \phi(i)$, the pair $\{\phi(i), \phi_i(v)\}$ is a pair of vertices of $M$. 
Since in $\operatorname{Ext}(M)$ the pair $\{\phi(i), \phi_i(v)\}$ is covered by an edge $e$, we can map the $k-2$ vertices of $Y_{v}^i$ to the $k-2$ vertices of $e\setminus \{\phi(i), \phi_i(v)\}$. 
Then the edge 
$\{i, v\} \cup Y_v^i$ is mapped to an edge of $\operatorname{Ext}(M)$. This gives a homomorphism from $\mathcal {T}^{(k)}_{1,k-2}(F)$ to $\operatorname{Ext}(M)$.
\end{proof}

It is well known that the Tur\'an density of a $k$-graph equals its homomorphic Tur\'an density. Hence, \cref{lem:extension} and \cref{auxiliary} yield that 
$$\pi(\mathcal {T}^{(k)}_{1,k-2}(F))\leq\pi(\operatorname{Ext}(\mathcal M))=\pi_{\lambda}(\mathcal{M}).$$
We next show that $\pi_\lambda(\mathcal M)\leq\Upsilon_k(\beta)$, and thus finish the proof.

Let $H$ be an $\mathcal M$-free $k$-graph, and fix a nonnegative weighting $(x_v)_{v\in V(H)}$ with $\sum_{v\in V(H)} x_v=1$. 
If $F$ has no homomorphism to $H$, then $k!\sum_{e\in E(H)}\prod_{v\in e}x_v\le\beta<\Upsilon_k(\beta)$ (otherwise a large weighted blow-up of $H$ would be
$F$-free and have edge density exceeding $\beta$). 
Now suppose that there is a homomorphism $\phi:F\to H$.  Since $H$ is $\mathcal M$-free, for some
$i\in[f]$ we have $F\not\to H-\phi(i)$. 

Write $w=\phi(i)$, $x=x_w$ and $y=1-x$. Since $F\not\to H-w$, the contribution of edges not containing $w$ in the sum $k!\sum_{e\in E(H)}\prod_{v\in e}x_v$ is at most $\beta y^k$. 
On the other hand, the contribution of edges containing $w$
is at most $kxy^{k-1}$.  Therefore, we have
\[
k!\sum_{e\in E(H)}\prod_{v\in e}x_v\le \beta y^k+kxy^{k-1}=\beta y^k+k(1-y)y^{k-1}.
\]
The right-hand side is maximized on $[0,1]$ at
$y=(k-1)/(k-\beta)$, and its maximum is
$\Upsilon_k(\beta)$.  Taking the supremum over all $\mathcal M$-free $k$-graphs and all
weightings proves the upper bound.
\end{proof}

\section{Algebraic complexity: Proof of \cref{thm:degree-high-degree}}\label{sec:support}

\subsection{Algebraic preliminaries}

Throughout this subsection, all fields are viewed as subfields of
$\mathbb C$. We write $\overline{\mathbb Q}$ for the algebraic closure
of $\mathbb Q$, and $\mathbb Q(X)$ for the field of rational functions
in the indeterminate $X$ with rational coefficients.

For an algebraic number $a$, let $\mathbb Q(a)$ denote the smallest
subfield of $\mathbb C$ containing both $\mathbb Q$ and $a$.
A \emph{number field} is a finite extension of $\mathbb Q$.
If $\mathbb K\subseteq \mathbb L$ are number fields, then
\[
  [\mathbb L:\mathbb K]:=\dim_{\mathbb K} \mathbb L
\]
denotes the degree of the field extension $\mathbb L/\mathbb K$. In particular, the
algebraic degree of $a$ over $\mathbb Q$ is
\[
  \deg_{\mathbb Q}(a)
  :=[\mathbb Q(a):\mathbb Q],
\]
equivalently, the degree of the minimal polynomial of $a$ over
$\mathbb Q$. We shall use the tower law
\[
  [\mathbb L:\mathbb Q]=[\mathbb L:\mathbb K][\mathbb K:\mathbb Q].
\]
All field embeddings into $\mathbb C$ are understood to fix
$\mathbb Q$.

We now give two arithmetic lemmas for a general rational function. The first says that an integer translate retains the current number field;
the second makes the next radical extension have full degree.

\begin{lemma}\label{lem:rational-primitive-shifts}
Let $R\in\mathbb Q(X)$ be non-constant, and let $\mathbb K=\mathbb Q(u)$ be a
number field.  Then, for all but finitely many integers $q$ for which
$R(u+q)$ is defined, $\mathbb Q\bigl(R(u+q)\bigr)=\mathbb K$.
\end{lemma}

\begin{proof}
The assertion is immediate if $\mathbb K=\mathbb Q$, so assume $\mathbb K\ne\mathbb Q$.
For an integer $q$ with $R(u+q)$ defined, put
$A_q=R(u+q)$.  If $\mathbb Q(A_q)\ne \mathbb K$, then separability gives two
distinct embeddings $\sigma,\tau:\mathbb K\hookrightarrow\mathbb C$ which agree
on $\mathbb Q(A_q)$.  Since $R$ has rational coefficients,
\[
        R\bigl(\sigma(u)+q\bigr)
        =R\bigl(\tau(u)+q\bigr).
\]
Fix $\sigma\ne\tau$.  If this equality held for infinitely many integers
$q$, then the two rational functions
$R(X+\sigma(u))$ and $R(X+\tau(u))$ would be identical.  Hence, with
$c=\sigma(u)-\tau(u)\ne0$, we would have
\[
        R(X+c)=R(X).
\]
A non-constant rational function over a field of characteristic zero has
no non-zero additive period: a finite pole would generate infinitely many
poles under translation by $c$, while a rational function with no finite
pole is a polynomial, and for a polynomial of degree $d\ge1$ the leading
term of $R(X+c)-R(X)$ has degree $d-1$ and non-zero coefficient.  Thus
only finitely many $q$ arise from each pair $\{\sigma,\tau\}$, and there are
only finitely many such pairs.
\end{proof}
If $R=N/D\in\mathbb Q(X)$, where $N,D\in\mathbb Q[X]$ are
coprime, then $c\in\overline{\mathbb Q}$ is a simple zero of $R$
if $N(c)=0$ and $N'(c)D(c)\ne 0$,
and is a simple pole of $R$ if $D(c)=0$
and $D'(c)N(c)\ne 0$.
For a number field $\mathbb K$, we write $\mathcal O_{\mathbb K}$ for its ring of
integers, and $v_{\mathfrak p}$ for the normalized additive valuation
associated with a nonzero prime ideal
$\mathfrak p\subseteq\mathcal O_{\mathbb K}$. We write $\mu_m$ for the set of
$m$-th roots of unity.
\begin{lemma}\label{lem:rational-radical-shifts}
Let $R\in\mathbb Q(X)$ have a simple zero or a simple pole at some
$c\in\overline{\mathbb Q}$.  Let $\mathbb K$ be a number field, let $u\in \mathbb K$, and
let $s\ge2$.  Then there are infinitely many integers $q\ge0$ such that
\[
        X^s-R(u+q)
\]
is irreducible over $\mathbb K$.
\end{lemma}

\begin{proof}
Write $R=N/D$ with coprime polynomials $N,D\in\mathbb Q[X]$.  In the
simple-zero case let $h=N$ and $g=D$; in the simple-pole case let $h=D$
and $g=N$.  Thus
$h(c)=0$ and $h'(c)g(c)\ne0$.
Choose a finite Galois extension $\mathbb M/\mathbb Q$ containing $\mathbb K$, $c$, and
the $2s$-th roots of unity.  By Chebotarev's theorem~\cite{neukirch},
there are infinitely many rational primes $p$ which split completely in
$\mathbb M$.  Choose one such $p$ after excluding the finitely many primes which
divide $2s$, the relevant denominators, or at which any of the displayed
non-zero quantities has bad reduction; also require that $u$ is integral
at every prime above $p$.

Fix a prime $\mathfrak P$ of $\mathbb M$ above $p$, and let
$\mathfrak p=\mathfrak P\cap\mathcal O_{\mathbb K}$.  Complete splitting identifies
the residue fields with $\mathbb F_p$.  Choose an integer $q_0$ such that
\[
        u+q_0\equiv c\pmod{\mathfrak P}.
\]
Then $h(u+q_0)\equiv0\pmod{\mathfrak P}$, while
$h'(u+q_0)g(u+q_0)\not\equiv0\pmod{\mathfrak P}$.  For
$q=q_0+tp$, Taylor expansion modulo $\mathfrak P^2$ shows that exactly one
class $t\pmod p$ makes $h(u+q)$ divisible by $\mathfrak P^2$.  Choose a
different class.  Every non-negative integer $q$ in the resulting class
modulo $p^2$ then satisfies
\[
        v_{\mathfrak p}\bigl(R(u+q)\bigr)=1
        \quad\text{in the simple-zero case},
\]
and
\[
        v_{\mathfrak p}\bigl(R(u+q)\bigr)=-1
        \quad\text{in the simple-pole case}.
\]
In either case, $R(u+q)$ is not a $d$-th power in $\mathbb K$ for any divisor
$d>1$ of $s$.  If $4\mid s$, then it is also not in $-4\mathbb K^4$, since
$p$ is odd and its $\mathfrak p$-valuation is $\pm1$.  Capelli's
criterion for binomials~\cite{lang-algebra} now implies that
$X^s-R(u+q)$ is irreducible over $\mathbb K$.
\end{proof}

\subsection{Proof of \cref{thm:degree-high-degree}}
Throughout this subsection, we set $k\geq 3$, $\ell\in\{1,\lceil k/2\rceil,\lceil k/2\rceil+1,\dots,k-2\}$ and $s=k-\ell$.
Recall that $u_{\ell}(x)=(1-x)^{-1/s}$, and let 
\[R_{\ell}(x)=
\begin{cases}
x^s+1,  &k/2\leq\ell\leq k-2,\\    
\frac{(sx^s+1)^s}{(sx^s+1)^s-(sx^s)^s}, & \ell=1.
\end{cases}\]
We begin by combining the two transfer functions of $\Pi^k_{\ell}$ given in \cref{tree-suspension} to obtain a new one. 
\begin{claim}\label{new-transfer}
For all integers $q\geq 1$, the function
\[
1-\frac{1}{R_{\ell}(u_\ell(x)+q)}
\]
is a transfer function of $\Pi^k_{\ell}$.
\end{claim}
\begin{proof}
Note that the composition of two transfer functions is still a transfer function.
Recall that $\Phi^q_{k,\ell}(x)=1-(q+u_{\ell}(x))^{-s}$. Then for $k/2\leq\ell\leq k-2$
\[
(\Psi\circ\Phi^q_{k,\ell})(x)=\frac{1}{2-\Phi^q_{k,\ell}(x)}=1-\frac{1}{(u_\ell(x)+q)^s+1}=1-\frac{1}{R_{\ell}(u_\ell(x)+q)}
\]
is a transfer function of $\Pi^k_{\ell}$.
Similarly, for $\ell=1$, we obtain 
\[
(\Upsilon_k\circ\Phi^q_{k,\ell})(x)=\left(\frac{k-1}{k-\Phi^q_{k,\ell}(x)}\right)^{k-1}=\left(\frac{s}{s+(q+u_{\ell}(x))^{-s}}\right)^{s}=1-\frac{1}{R_{\ell}(u_\ell(x)+q)}.
\]
\end{proof}

Thus for any integer $q\geq 1$, the new transfer function in \cref{new-transfer} sends a single $k$-graph with $\ell$-degree Tur\'an density $1-u^{-s}$ to a single $k$-graph with $\ell$-degree Tur\'an density $1-\frac{1}{R_\ell(u+q)}$. We now give the proof of \cref{thm:degree-high-degree}.

\begin{proof}[Proof of \cref{thm:degree-high-degree}]
Let $F_0$ be the single-edge $k$-graph. Then $\pi_\ell(F_0)=0$, so the positive root of $\pi_\ell(F_0)=1-x^{-s}$ is $u_0=1$, and $\mathbb K_0:=\mathbb Q(u_0)=\mathbb Q$.
We construct recursively $k$-graphs $F_i$ and positive real numbers $u_i$ such that
\[
        \pi_\ell(F_i)=1-u_i^{-s},
        \qquad
        \mathbb K_i:=\mathbb Q(u_i),
        \qquad
        [\mathbb K_i:\mathbb Q]=s^i.
\]
Moreover, for $i\ge1$ we will have $\mathbb Q(\pi_\ell(F_i))=\mathbb K_{i-1}$, so $\deg_{\mathbb Q}\pi_\ell(F_i)=s^{i-1}$. Since $s\ge2$ in all cases covered by the theorem, this proves the result. 
  
We first verify that $R_{\ell}$ satisfies the hypothesis of \cref{lem:rational-radical-shifts}.  For $k/2\leq\ell\leq k-2$,
$R_{\ell}(X)=X^s+1$ has a simple zero at any $c$ with $c^s=-1$.  For $\ell=1$, write $R_\ell(X)=\frac{N(X)}{D(X)}$ with $N(X)=(sX^s+1)^s$ and $D(X)=(sX^s+1)^s-(sX^{s})^s$.
Choose $\zeta\in\mu_s\setminus\{1\}$ and $c$ with
$c^s=\frac{1}{s(\zeta-1)}$.
Then $s c^s+1=\zeta s c^s$. Hence $D(c)=0$ and $N(c)\ne0$.  Moreover,
$D'(c)=s^{s+2}c^{s^2-1}(\zeta^{-1}-1)\ne0$.
Thus $R_\ell$ has a simple pole at $c$.

Now we construct the sequence $(F_i)_{i\geq 1}$ recursively, and suppose that we have obtained the $k$-graph $F_i$ (initially $i=0$). 
By \cref{lem:rational-primitive-shifts,lem:rational-radical-shifts}, there is an integer
$q_i\ge1$ such that $\mathbb Q(R_\ell(u_i+q_i))=\mathbb K_i$ and $X^s-R_\ell(u_i+q_i)$ is irreducible over $\mathbb K_i$.
Recall that $\pi_\ell(F_i)=1-u_i^{-s}$.
Then, by \cref{new-transfer}, we can choose a $k$-graph $F_{i+1}$ with 
\[
\pi_{\ell}(F_{i+1})=1-\frac{1}{R_\ell(u_i+q_i)}.
\]
Let $u_{i+1}$ be the positive real root of $x^s-R_\ell(u_i+q_i)=0$.  
Then $\pi_\ell(F_{i+1})=1-u_{i+1}^{-s}$.
Since $R_\ell(u_i+q_i)$ generates $\mathbb K_i$ and $X^s-R_\ell(u_i+q_i)$ is irreducible over $\mathbb K_i$, we obtain
$[\mathbb Q(u_{i+1}):\mathbb K_i]=s$.
Thus, setting $\mathbb K_{i+1}=\mathbb Q(u_{i+1})$ yields that
\[
[\mathbb K_{i+1}:\mathbb Q]=[\mathbb K_{i+1}:\mathbb K_{i}][\mathbb K_{i}:\mathbb Q]=s^{i+1}.
\]
Note also that
\[
\mathbb Q\bigl(\pi_\ell(F_{i+1})\bigr)=\mathbb Q\left(1-\frac{1}{R_\ell(u_i+q_i)}\right)=\mathbb Q(R_\ell(u_i+q_i))=\mathbb K_i,
\]
so $\pi_\ell(F_{i+1})$ has algebraic degree exactly $s^i$ over
$\mathbb Q$.
\end{proof}


\section{Concluding remarks}\label{remark}

In this paper we introduced tree suspensions as a unified mechanism for producing transfer functions in degree Tur\'an spectra.  More precisely, our construction may be viewed as a family of $(\ell,t)$-tree suspensions, where $0\le t\le s-1$.
For every value of $t$ in this range, there is a natural two-part construction giving the lower bound $Q_{s,t+1}(\gamma)$ for the corresponding $\ell$-degree Tur\'an density.  
We prove that this lower bound is sharp in the cases $t=0$, and also in the case $t=k-\ell-1$ when either $\ell\ge k/2$ or $\ell=1$.  These sharpness results are exactly what give the transfer functions used in this paper.

It is natural to ask whether the same two-part construction is always optimal.

\begin{conjecture}
Let $1\le \ell<k$ and $F$ be a $k$-graph with $\pi_{\ell}(F)=\beta$. Then $\pi_{\ell}(\mathcal T^{(k)}_{\ell,t}(F))=Q_{s,t+1}(\gamma)$ for every $0\le t< s$, where $s=k-\ell$ and $\gamma$ is the unique solution of $Q_{s,t+1}(x)=Q_{s,t+1-\ell}(x)-(1-\beta)x^s$.
\end{conjecture}

If this conjecture holds, then the tree-suspension framework would yield additional transfer functions in the remaining ranges of $\ell$.  In particular, for those degree spectra not covered by our present algebraic-degree theorem, one would obtain two distinct transfer functions acting on the same spectrum.  Iterating these functions should then produce elements of arbitrarily large algebraic degree, extending our high-degree result to all remaining non-codegree degree Tur\'an spectra.

The codegree case appears to be fundamentally different.  When $\ell=k-1$, the parameter range for $t$ collapses to the single value $t=0$.  Thus even if the above conjecture were proved in full generality, the present tree-suspension framework would still produce only one transfer function for the single codegree spectrum.  Since one transfer function of the type obtained here is not enough to force unbounded algebraic degree, the following problem remains especially intriguing.

\begin{problem}
For $k\ge3$, does the single codegree Tur\'an spectrum $\Pi^k_{k-1}$ contain algebraic numbers of arbitrarily large degree? 
\end{problem}

A positive answer would require a genuinely new construction producing a transfer function different from the one arising from tree suspensions.  Such a construction would have to be specific to higher-uniformity codegree problems.  Indeed, in the graph case $k=2$, codegree coincides with ordinary graph density, and the possible values are completely determined by the Erd\H{o}s--Stone--Simonovits theorem; in particular, no irrational values occur.  This raises the possibility that the single codegree spectrum may have a more rigid arithmetic structure than other degree spectra, even though finite-family codegree spectra in uniformity at least three are much richer and can realize all rational values \cite{GaoPikhurkoRongSun}.

A broader question is whether transfer functions can be used not only to generate new examples, but also to organize the entire spectrum.  Given a collection of transfer functions acting on $\Pi^k_\ell$, one may consider the semigroup that they generate and the orbits of known seed densities.  It is probably too optimistic to expect that every element of $\Pi^k_\ell$ is obtained exactly from a small set of seeds in this way.  Nevertheless, the orbit structure and its closure seem worth studying: they may explain why accumulation points and high algebraic complexity arise so naturally, and they may provide a systematic way of comparing different regions of the spectrum.

Finally, the tree-suspension method is not restricted to single forbidden hypergraphs.  The same two-part lower-bound constructions and the corresponding upper-bound questions make sense for finite forbidden families, and more generally for arbitrary families.  In the finite-family setting this should lead to transfer functions for finite-family degree spectra.  In the infinite-family setting one can also allow several input densities simultaneously, leading naturally to multivariable transfer functions.  Such multivariable operations may provide a useful framework for studying the much richer topology and arithmetic of family degree Tur\'an spectra.


\medskip
\noindent{\bf Acknowledgments and AI disclosure.~}
At an early stage of this project, the authors used AI as a brainstorming aid. In particular, one such discussion suggested a prototype of the $(2,0)$-tree suspension $\mathcal T^{(k)}_{2,0}(F)$, which the authors subsequently developed into the general tree-suspension framework used in this paper. The authors also used AI assistance on technical algebraic lemmas in Section~\ref{sec:support}.

\bibliographystyle{abbrv}
\bibliography{reference-suspension}

\end{document}